\documentclass[english,a4paper]{article}

\usepackage[margin=3.8cm]{geometry}

\usepackage{babel}

\usepackage{amssymb,amsfonts,amsmath,amsthm}

\usepackage{enumitem}

\usepackage{tikz-cd}

\usepackage[T1]{fontenc}
\usepackage[utf8]{inputenc}
\usepackage{csquotes}

\DeclareMathOperator{\rk}{rk}
\DeclareMathOperator{\diag}{diag}
\DeclareMathOperator{\End}{End}
\DeclareMathOperator{\Aut}{Aut}
\DeclareMathOperator{\GL}{GL}
\DeclareMathOperator{\Irr}{Irr}
\DeclareMathOperator{\ind}{ind}
\DeclareMathOperator{\res}{res}
\DeclareMathOperator{\Hom}{Hom}
\DeclareMathOperator{\Gal}{Gal}

\newtheorem{theorem}{Theorem}[section]
\newtheorem{main_theorem}{Theorem}
\newtheorem{corollary}[theorem]{Corollary}
\newtheorem{lemma}[theorem]{Lemma}
\newtheorem{proposition}[theorem]{Proposition}
\theoremstyle{definition}
\newtheorem{definition}[theorem]{Definition}
\theoremstyle{remark}
\newtheorem{remark}[theorem]{Remark}
\newtheorem*{remark*}{Remark}
\newtheorem{example}[theorem]{Example}

\newcommand{\Z}{{\mathbb Z}}
\newcommand{\Q}{{\mathbb Q}}
\newcommand{\R}{{\mathbb R}}
\newcommand{\CC}{{\mathbb C}}
\newcommand{\HH}{{\mathbb H}}
\newcommand{\N}{{\mathbb N}}
\newcommand\X{\ensuremath{\mathcal X}}

\usepackage[backend=biber,arxiv=pdf,isbn=false,doi=false,url=false,eprint=false]{biblatex}
\bibliography{bibl.bib}

\title{Complex Vasquez invariant}
\author{Anna Gąsior and Rafał Lutowski}
\date{October 17, 2023}

\begin{document}

\maketitle

\renewcommand{\thefootnote}{\fnsymbol{footnote}}
\renewcommand{\thefootnote}{\arabic{footnote}}

\begin{abstract}
In 1970 Vasquez proved that to every finite group $G$ we can assign a natural number $n(G)$ with the property that every flat manifold with holonomy $G$ is a total space of a fiber bundle, with the fiber being a flat torus and the base space -- a flat manifold of dimension less than or equal to $n(G)$. In particular, this means that the characteristic algebra of any flat manifold with holonomy $G$ vanishes in dimension greater than $n(G)$. We define a complex analog of Vasquez invariant, in which finite groups are considered as holonomy groups of compact flat K\"ahler manifolds.
\end{abstract}

\section{Introduction}

Let $X$ be an $n$-dimensional \emph{flat manifold}, i.e. a closed connected Riemannian manifold with vanishing sectional curvature. It is well known (see \cite{Sz12}) that a torsion-free group $\Gamma := \pi_1(X)$ defines
a short exact sequence
\begin{equation}
\label{eq:exact}
0 \longrightarrow L \longrightarrow\Gamma \stackrel{p}{\longrightarrow} G \longrightarrow 1,
\end{equation} 
where the free abelian group $L$, of rank $n$, is the unique maximal abelian normal subgroup of $\Gamma$ and $G$ is a finite group. We shall call $\Gamma$ a \emph{Bieberbach group} of dimension $n$ with the \emph{holonomy group} $G$.
By conjugation in $\Gamma,$ the above extension defines a $G$-lattice structure on $L$. The corresponding representation $h \colon G\to \GL(L)$ is called the \emph{integral holonomy representation} of $\Gamma$.

In \cite{V70}, Vasquez (see also \cite{CW89} and \cite{Sz97}) assigned to every finite group $G$ a natural number $n(G)$. He proved that if $\Gamma$ is a Bieberbach group with holonomy $G$, as in \eqref{eq:exact}, then there exists a normal subgroup $N$ of $\Gamma$ such that $N \subset L$ and the quotient group $\Gamma/L$ is a Bieberbach group of dimension less than or equal to $n(G)$.  

Vasquez showed that $n(G) = 1$ for every cyclic group $G$ of a prime order. In \cite{CW89} Cliff and Weiss proved that if $G$ is a $p$-group, then $n(G) = \sum_{C\in \mathcal{X}}[G:H]$, where $\mathcal{X}$ is a set of representatives of the conjugacy classes of subgroups of $G$ of prime order. Moreover $n(A_5) = 16$, see \cite{CZ96}. The articles \cite{Sz97} and \cite{F14} give the full classification of finite groups with Vasquez invariant equal to $1$ and $2$, respectively. In \cite{CP22} the authors consider the Vasquez invariant for elementary abelian groups. 

The goal of the paper is a description of the complex analog of the Vasquez invariant. A \emph{compact flat K\"ahler} or \emph{generalized hyperelliptic} manifold $X$ of dimension $n$ is defined as a quotient of a compact complex $n$-torus by a free action of a finite group. The fundamental group $\Gamma = \pi_1(X)$ is a Bieberbach group of dimension $2n$. In particular, $\Gamma$ may be realized as a subgroup of $U(n) \ltimes \CC^n \subset O(2n) \ltimes \R^{2n}$, see \cite[Proposition 7.1]{Sz12}. The classes of generalized hyperelliptic and aspherical K\"ahler manifolds coincide (see \cite[Theorem 1]{CC17}). We will follow \cite{BR11} and call fundamental groups of compact K\"ahler flat manifolds \emph{aspherical K\"ahler groups}. It is well known that any finite group is a holonomy group of a generalized hyperelliptic manifold (see Remark \ref{remark:holonomy-of-kahler}).

\begin{remark*}
Unless stated otherwise, whenever we say about the dimension of an aspherical K\"ahler group, we mean its complex dimension.
\end{remark*}

\begin{main_theorem}
\label{theorem:complex-vasquez}
Let $G$ be a finite group. Then there is an integer $n_{\CC}(G)$ such that if $\Gamma$ is an aspherical K\"ahler group of dimension $n$ and with holonomy group $G$, then the maximal abelian subgroup $L \subset\Gamma$ contains a subgroup $M$, normal in $\Gamma$, such that $\Gamma/M$ is an aspherical K\"ahler group of dimension less than or equal to $n_{\CC}(G).$
\end{main_theorem}

Although it is not as direct as in the real case, using the above theorem we can also formulate a result concerning characteristic classes of compact flat K\"ahler manifolds.

\begin{main_theorem}
\label{theorem:chern-classes}
Let $X$ be a generalized hyperelliptic manifold with the holonomy group $G$. Then for every integer $i > n_{\CC}(G)$ the $i$-th Chern class of $X$ is zero.
\end{main_theorem}

The structure of the paper is as follows. In Section 2 we provide a modified proof of the original Vasquez result, which was suggested by Cliff and Weiss in \cite{CW89}. This gives us a better estimate of the invariant and allows us to understand the idea standing behind its complex analog. The next section deals with essentially complex modules. Although the topic was presented in \cite{Jo90}, we give a slightly different approach here, suited for the proof of Theorem \ref{theorem:complex-vasquez}, presented along with some examples in Section 4. The algebraic approach from this section is necessary, but not sufficient if one would like to consider characteristic classes of holomorphic tangent bundles. Hence in Section 5, we give a criterion for a smooth map that arises from algebraic construction to be a holomorphic one. This condition may require the complex structure to be changed. Section 6 describes how to make this change while keeping the holomorphic tangent bundle unchanged (up to isomorphism of course). These results are then applied to show Theorem \ref{theorem:chern-classes} in the last section of the article.

\section{Modified Vasquez construction}

Let $\Gamma$ be a Bieberbach group defined by the short exact sequence \eqref{eq:exact}. A cohomology class $\alpha \in H^2(G,L)$ that corresponds to this extension is called \emph{special}. In fact, a faithful $G$-lattice and a special element make all that is needed to define a Bieberbach group. Hence, we can formulate Vasquez's theorem in the module-theoretic setting:

\begin{theorem}[{\cite[Theorem 3.6]{V70}}]
\label{thm:vasquez}
For any finite group $G$ there exists a number $n(G)$ such that if $L$ is a faithful $G$-lattice with a special element $\alpha$ then there exists a $\Z$-pure submodule $N$ of $L$ such that:
\begin{enumerate}[label=\rm(\roman*)]
	\item $\rk_{\Z}(L/N) \leq n(G)$,
	\item $\nu_*(\alpha)$ is special,
\end{enumerate}
where $\nu \colon L \to L/N$ is the natural homomorphism.
\end{theorem}

\begin{remark}
\label{remark:pure_submodule}
We consider quotient $G$-lattices which are free abelian groups. In other words, we demand from sublattices to be $\Z$-pure submodules (see \cite[(16.15)]{CR62}). It is easy to check that an intersection of a finite number of $\Z$-pure sublattices of a given lattice is again $\Z$-pure sublattice.
\end{remark}

In his proof, Vasquez focused on giving an upper bound on $n(G)$. In \cite{CW89} Cliff and Weiss, by using different methods, achieved a better estimate of it. They remarked that a slight modification of the proof of Vasquez could be used to get their result. Since this is of much importance to our further considerations, we give a proof of Vasquez's theorem using the hint given by Cliff and Weiss.

Before giving the proof, let us note that estimating $n(G)$ from above is not very precise. It is natural to demand from $n(G)$ to be as small as possible. Hence we give the following definition of the Vasquez number.

\begin{definition}
\label{definition:vasquez_number}
For any finite group $G$ a \emph{Vasquez number} $n(G)$ is the smallest natural number which satisfies the conclusion of Theorem \ref{thm:vasquez}.
\end{definition}

\begin{proof}[Proof of Theorem \ref{thm:vasquez}]
Recall that $\alpha \in H^2(G,L)$ is special if and only if its restriction to every nontrivial subgroup of $G$ is non-zero. Since the restriction on chains of subgroups of $G$ is transitive, using in addition the standard action of $G$ on those restrictions (see \cite[page 65]{Sz12}, \cite[page 168]{Ch86}) one easily gets that $\alpha$ is special if and only if
\[
\forall_{H \in \X} \res_H \alpha \neq 0,
\]
where $\X$ is a set of representatives of conjugacy classes of subgroups of $G$ of prime order.

Now take $H \in \X$. Since $\res_H \alpha \neq 0$, we get that as an $H$-module, $L$ has a direct summand $L_0$ of rank $1$, which admits the trivial $H$-action. Hence we have 
\begin{equation}
\label{eq:res_decomposition}
\res_{H} L = L_0 \oplus L_0'.
\end{equation}
Furthermore, this decomposition can be taken in such a way that if
\[
\res_H \alpha = \alpha_0 + \alpha_0' \in H^2(H,L_0) \oplus H^2(H,L_0')
\]
then $\alpha_0 \neq 0$. Now, $L_0'$ may be not a $G$-submodule of $L$, but
\[
L_H' := \cap_{g \in G} gL_0'
\]
is one. Moreover, since $gL_0'$ is a $\Z$-pure submodule of the free $\Z$-module $L$ (see \cite[Theorem 16.16]{CR62}), by Remark \ref{remark:pure_submodule} $L/L_H'$ is a free abelian group. By the first isomorphism theorem for modules, there exists the unique map $r$, such that the following diagram commutes
\begin{equation}
\label{eq:restriction}
\begin{tikzcd}
\res_H L \arrow[rdd,"\pi"] \arrow[rr, "\res_H p^{(H)}"] & & \res_H L/L_H' \arrow[ldd, "r"] \\
\\
& L_0
\end{tikzcd}
\end{equation}
where $p^{(H)} \colon L \mapsto L/{L_H'}$ is the natural mapping and $\pi$ is the projection corresponding to decomposition \eqref{eq:res_decomposition}. Hence, if 
\[
\res_H p^{(H)}_*(\alpha) = (\res_H p^{(H)})_*(\res_H \alpha) = 0,
\]
then 
\[
\pi_*(\res_H \alpha) = \alpha_0 = 0,
\] 
which contradicts our assumptions.

In the Frobenius reciprocity
\[
\Hom_H(\res_H L, L_0) \cong \Hom_G (L, \ind_H^G L_0),
\]
the kernel of the map corresponding to $\pi$ is equal to $L_H'$. Hence we get
\[
\rk_{\Z}(L/L'_H) \leq \rk_{\Z} \ind_H^G L_0 = [G:H].
\]

Summarizing, for a group $H \in \X$ we have constructed a $G$-sublattice $L_H'$ of $L$ such that $L/L_H'$ is again a $G$-lattice of $\Z$ rank bounded by the index of $H$ in $G$ and that the restriction to $H$ of class induced by $\alpha$ is nonzero.

Let $N := \cap_{H \in \X} L_H'$. This is a $G$-sublattice of $L$ and since $\X$ is finite, by Remark \ref{remark:pure_submodule}, $L/N$ is torsion-free. Let $\nu \colon L \to L/N$ be the natural homomorphism. Making diagrams similar to \eqref{eq:restriction} we find that $\nu_*(\alpha)$ is special. Moreover, we have
\[
\rk_{\Z} (L/N) \leq \sum_{H \in \X} \rk_{\Z}(L/L'_H) \leq \sum_{H \in \X} [G:H].
\]
Therefore, we obtain an estimate of $n(G)$ given in \cite[Corollary on page 125]{CW89}:
\[
n(G) \leq \sum_{H \in \X} [G:H].
\]
\end{proof}

\section{Essentially complex modules}

Throughout this section $G$ will denote a finite group, and $K$ -- the ring $\Z$, or the field $\Q$ or $\R$.

\begin{remark}
\label{remark:extending_field}
Since the extending of the ring/field of scalars will be frequently used throughout the paper, for any subfield $F$ containing $K$ we introduce the following notation:
\[
L^F := F \otimes_K L.
\]
\end{remark}

\begin{definition}
Let $V$ be a $K G$-module. \emph{Complex structure} on $V$ is a map $J \in \End_{\R G}(V^\R)$ such that $J^2 = -id$. A module admitting a complex structure is called \emph{essentially complex}.
\end{definition}

By Johnson, we have the following criterion for a module being essentially complex for the case $K = \R$.

\begin{theorem}[{\cite[Proposition 3.1]{Jo90}}]
\label{theorem:complex-r-module}
Let $V$ be an $\R G$-module. The following are equivalent:
\begin{enumerate}
	\item $V$ is essentially complex.
	\item Every homogeneous component of $V$ is essentially complex.
	\item Every absolutely irreducible component of $V$ occurs with even multiplicity.
\end{enumerate}
\end{theorem}

\begin{remark}
\label{remark:holonomy-of-kahler}
Let $G$ be a finite group. By the Auslander and Kuranishi theorem, $G$ is a holonomy group of some Bieberbach group (see \cite[Theorem III.5.2]{Ch86}). Hence there exists a faithful $G$-lattice $L$ and a special element $\alpha \in H^2(G,L)$. Then $L \oplus L$ is also a faithful $G$ lattice, which has a special element, for example $\alpha = \alpha + 0 \in H^2(G,L)\oplus H^2(G,L)$. Hence, by \cite[Proposition 7.1]{Sz12} every finite group is a holonomy group of an aspherical K\"ahler group.
\end{remark}

Let $\Irr(G)$ be the set of complex irreducible characters of $G$ and $\chi \in \Irr(G)$. The Frobenius-Schur indicator 
\[
\nu_2(\chi) = \frac{1}{|G|} \sum_{g \in G} \chi(g^2),
\]
which takes values in $\{-1,0,1\}$ establishes a well-known connection between $\R G$ and $\CC G$-modules (see \cite[Section II.13.2]{S75}). We put it here in a more general context, first stating the following lemma. For characters $\chi_1,\chi_2$ of the group $G$, by $(\chi_1,\chi_2)$ we denote the usual inner product of $\chi_1$ and $\chi_2$.

\begin{lemma}
\label{lemma:fsi}
Let $V$ be a simple $K G$-module with the character $\chi_V$.  Let $\chi_s \in \Irr(G)$ be such that $(\chi_V,\chi_s) \neq 0$. Then 
\[
(\chi_V,\chi) \neq 0 \Rightarrow \nu_2(\chi) = \nu_2(\chi_s)
\]
for every $\chi \in \Irr(G)$.
\end{lemma}
\begin{proof}
Let $\chi$ be as in the statement of the lemma. By \cite[Corollary (10.2)]{I76} there exists an automorphism $\sigma \in \Gal(K(\chi)/K)$ such that $\chi = \sigma\chi_s$, where $K(\chi)$ is the extension of $K$ by values of $\chi$. We get
\[
\nu_2(\chi) = \frac{1}{|G|} \sum_{g \in G} \sigma\chi_s(g^2) = \sigma\left(\frac{1}{|G|} \sum_{g \in G} \chi_s(g^2)\right) = \sigma(\nu_2(\chi_s)) = \nu_2(\chi_s),
\]
since $\nu_2(\chi_s)$ is an integer.
\end{proof}

The above lemma justifies the following definition.

\begin{definition}
Let $V$ be a $KG$-module with character $\chi_V$. Let $\chi_s \in \Irr(G)$ be any character such that $(\chi_V,\chi_s) \neq 0$. We say that $V$ is of
\begin{itemize}
	\item real/$\R$ type if $\nu_2(\chi_s) = 1$,
	\item complex/$\CC$ type if $\nu_2(\chi_s) = 0$,
	\item quaternionic/$\HH$ type if $\nu_2(\chi_s) = -1$.
\end{itemize}
\end{definition}

\begin{remark}
\label{remark:complex-quaternionic-case}
Note that an irreducible $\R G$-module $V$ is of $F$-type if and only if we have an isomorphism of $\R$-algebras:
\[
\End_{\R G}(V) \cong F.
\]
Hence by Theorem \ref{theorem:complex-r-module} and Lemma \ref{lemma:fsi} irreducible $KG$-modules of type $\CC$ or $\HH$ are essentially complex.
\end{remark}

Using Lemma \ref{lemma:fsi} again, we state the following definition.

\begin{definition}
Let $V$ be an irreducible $\Q G$-module with character $\chi_V$. Let $\chi_s \in \Irr(G)$ with $(\chi_V, \chi_s) \neq 0$. Define
\[
m(V) := m_{\Q}(\chi_s),
\]
where $m_{\Q}(\chi_s)$ is the Schur index of $\chi_s$ over the rationals.
For an irreducible $G$-lattice $L$ we define 
\[
m(L) := m(L^\Q).
\]
\end{definition}

We immediately get a complex structure criterion on irreducible lattices.

\begin{proposition}
Let $L$ be an irreducible $G$-lattice. The following are equivalent:
\begin{enumerate}
    \item $L$ is essentially complex.
    \item $L^\Q$ is essentially complex.
    \item $L$ is of type $\CC,\HH$ or $m(L)$ is even.
\end{enumerate}
\end{proposition}

\begin{proof}
Equivalence of 1 and 2 is by definition. By Remark \ref{remark:complex-quaternionic-case} it is enough to consider the case when $V:=L^\Q$ is of real type. Let $m=m(V)$, $\chi_V$ be the character of $V$, $\chi_s \in \Irr(G)$ be one character such that $(\chi_V,\chi_s) \neq 0$ and $\mathcal G$ be the class of Galois conjugates of $\chi_s$ in $\Gal(\Q(\chi_s)/\Q)$. We get
\begin{equation}
\label{eq:decomposition-m-v}
\chi_V = \sum_{\chi \in \mathcal G} m \chi.
\end{equation}
Since every $\chi \in \mathcal G$ is a character of an absolutely irreducible real representation, by Theorem \ref{theorem:complex-r-module} we get that $m$ is an even number.
\end{proof}

In the spirit of \cite[Theorem 3.3]{Jo90} we can state the following proposition.

\begin{theorem}
A $G$-lattice $L$ is essentially complex if and only if every simple component $V$ of $L^\Q$ of type $\R$ with odd $m(V)$ occurs with even multiplicity.
\end{theorem}
\begin{proof}
It is enough to note that if a simple module $V$ with odd $m=m(V)$ occurs in $L^\Q$ with multiplicity $k$, then $kV$ is isomorphic to a homogeneous component of $L^\Q$ and the formula \eqref{eq:decomposition-m-v} implies that
\[
k \cdot \chi_V = \sum_{\chi \in \mathcal G} km \chi.
\]
Hence, using Theorem \ref{theorem:complex-r-module}, if we want $L$, or equivalently $kV$, to be essentially complex, we need $k$ to be even.
\end{proof}

\section{Complex Vasquez invariant}

In this section, we will prove Theorem \ref{theorem:complex-vasquez}. Similar as in the real case, we state first the definition which follows from the theorem.

\begin{definition}
\label{definition:complex_vasquez_number}
For any finite group $G$ a \emph{complex Vasquez number} $n_\CC(G)$ is the smallest natural number which satisfies the conclusion of Theorem \ref{theorem:complex-vasquez}.
\end{definition}

\begin{proof}[Proof of Theorem \ref{theorem:complex-vasquez}]
Let $\Gamma$ be defined by the short exact sequence \eqref{eq:exact}.
By Theorem \ref{thm:vasquez} there exists a $G$-submodule $N \subset L$, i.e. a normal subgroup $N$ of $\Gamma$, such that $\Gamma/N$ is a Bieberbach group of dimension less than or equal to $n(G)$.

Let 
\[
L^\Q/N^\Q = m_1 V_1\oplus\ldots\oplus m_k V_k
\]
be a decomposition such that $V_1,\ldots,V_k$ are irreducible, pairwise non-iso\-mor\-phic $\Q G$-modules. 
Assume that for some $1 \leq i \leq k$, $m_i V_i$ is not essentially complex. This means that $V_i$ is of real type and that both $m_i'$ and $m(V_i)$ are odd. Since the multiplicity of $V_i$ in $L$ is even, $V_i$ must be a composition factor of $N^\Q$.
This shows that we can find a maximal submodule $M'$ of $N^\Q$ such that both $M'$ and $L^\Q/M'$ are essentially complex. In particular, we get
\begin{equation}
\label{equation:upper_bound}
L^\Q/M' = n_1' V_1\oplus\ldots\oplus n'_k V_k,
\end{equation}
where
\begin{equation}
\label{equation:multiplicity}
n_i' = 
\left\{
\begin{array}{ll}
m_i+1 & \text{if } m_i' V_i \text{ is not essentially complex},\\
m_i   & \text{otherwise}.
\end{array}
\right.
\end{equation}
We conclude that
\begin{equation}
\label{equation:dimensions}
\dim_\Q(L^\Q/M') \leq 2 \dim_\Q(L^\Q/N^\Q) \leq 2n(G).
\end{equation}
Now, let
\[
M:=L\cap M'.
\]
We easily get that $M$ is a $\Z$-pure submodule of $L$ with $\Z$-rank equal to $\dim_\Q(M')$ (see \cite[(16.19)]{CR62}).
Obviously $M \subset L \cap N^\Q = N$ and we have
\[
\Gamma/N\cong {\Gamma/M}{\Big/}{N/M}.
\]
Let $\gamma\in\Gamma\setminus  M$ be such that $\gamma^k\in M$ for some positive integer $k$. Then $\gamma\not\in N$ since $N/M$, as a subgroup of $L/M$, is torsion-free. We get that $\gamma\in \Gamma\setminus N$, but $\gamma^k\in N$, so $\gamma N$ is an element of finite order in $\Gamma/N$, which contradicts our assumptions on $\Gamma/N$. We get that $\Gamma/M$ is torsion-free of real dimension $\rk_\Z(L/M) = \dim_\Q(L^\Q/M')$, hence using inequality \eqref{equation:dimensions} we prove our claim, showing in particular that
\[
n_{\CC}(G)\leq n(G).
\]
\end{proof}

With Theorem \ref{theorem:complex-vasquez} we have got an upper bound for complex Vasquez number. We can in fact show more:

\begin{lemma}
\label{lemma:lower_bound}
Let $G$ be a finite group. Then
\[
n(G)/2 \leq n_{\CC}(G).
\]
\end{lemma}

\begin{proof}
Let $G$ be a finite group for which $n_{\CC}(G) < n(G)/2$. Let $\Gamma$ be a Bieberbach group, defined by the short exact sequence \eqref{eq:exact}. The lattice $L$ does not have to be essentially complex, but -- arguing as in the proof of Theorem \ref{theorem:complex-vasquez} -- there exists a $G$-lattice $L'$ of minimal $\Z$-rank such that $L \oplus L'$ admits a complex structure. In particular we have $\rk_\Z L' \leq \rk_\Z L$. Moreover, if $\alpha \in H^2(G,L)$ defines the group $\Gamma$, then 
\[
\alpha + 0 \in H^2(G,L) \oplus H^2(G,L')
\]
defines an aspherical K\"ahler group $\Gamma'$. By our assumption there exists an essentially complex $G$-submodule $M \subset L \oplus L'$, such that $\Gamma'/M$ is torsion-free of real dimension 
\[
\rk_\Z (L \oplus L')/M \leq 2n_{\CC}(G) < n(G).
\]
The minimality of $L'$ implies that it is not an essentially complex module, hence $M^\Q\cap L^\Q \neq 0$ and in particular $M \cap L \neq 0$.
Now we can argue as in the proof of Theorem \ref{thm:vasquez}. By Remark \ref{remark:pure_submodule} $(M \cap L) \oplus (M \cap L')$ is a $G$-sublattice and a pure $\Z$-submodule of $L \oplus L'$. Using a composition of maps
\[
L \oplus L' \to L/(M \cap L) \oplus L'/(M \cap L') \to (L \oplus L')/M
\]
we get that the image of $\alpha$ in $H^2(G,L/(M \cap L))$ is special and hence $\Gamma/(M \cap L)$ is a Bieberbach group of dimension 
\[
\rk_\Z L/(M \cap L) \leq \rk_\Z (L \oplus L')/N < n(G),
\]
which contradicts the minimality of Vasquez number (see Definition \ref{definition:vasquez_number}).
\end{proof}

In some cases, the following lemma can give us a better estimate of complex Vasquez invariant than the proof of Theorem \ref{theorem:complex-vasquez}.

\begin{lemma}
Let $G$ be a finite group and $\Irr_1(G) := \{ x \in \Irr(G) : \nu_2(\chi)=1 \}$. Then
\[
n_{\CC}(G) \leq \frac{1}{2} \left( n(G) + \sum_{\chi \in \Irr_1(G)} m_\Q(\chi)\chi(1) \right).
\]
\end{lemma} 

\begin{proof}
Let $\mathcal L$ be the set of representatives of isomorphism classes of irreducible $\Q G$-modules of real type. Equations \eqref{equation:upper_bound} and \eqref{equation:multiplicity} show that
\[
2n_{\CC}(G) \leq n(G) + \sum_{L \in \mathcal L} \dim_\Q L.
\]
Since every simple $\CC G$-module can be a component of only one of the modules $L^\CC$, for $L \in \mathcal L$, we get that
\[
\sum_{L \in \mathcal L} \dim_\Q L = \sum_{\chi \in \Irr_1(G)} m_\Q(\chi)\chi(1).
\]
\end{proof}

To sum up, we have the following bounds on complex Vasquez invariant.

\begin{proposition}
\label{proposition:bounds}
Let $G$ be a finite group. Then
\begin{enumerate}
    \item $n(G)/2 \leq n_{\CC}(G)$.
    \item $n_{\CC}(G) \leq n(G)$.
    \item $\displaystyle n_{\CC}(G) \leq \frac{1}{2} \left( n(G) + \sum_{\chi \in \Irr_1(G)} m_\Q(\chi)\chi(1) \right)$. \label{proposition:bounds:schur}
\end{enumerate}
\end{proposition}

\begin{corollary}
\label{corollary:odd_order}
Let $G$ be a group of odd order. Then 
\[
n_{\CC}(G) = \lfloor( n(G) + 1 )/2\rfloor.
\]
\end{corollary}
\begin{proof}
It is enough to note that in the case of odd order group $G$, the set $\Irr_1(G)$ consists only of the trivial character and then use lower and upper bounds of Proposition \ref{proposition:bounds}.
\end{proof}

\begin{example}
Let $G=C_3^k$ be an elementary abelian $3$-group of rank $k$. By \cite[Corollary on page 126]{CW89} $n(G) = 3^{k-1}(3^k-1)/2$ and by Corollary \ref{corollary:odd_order} we get that
\[
n_{\CC}(C_3^k) =
\left\{
\begin{array}{lll}
n(C_3^k)/2 && \text{if $k$ is even}, \\
(n(C_3^k)+1)/2 && \text{if $k$ is odd}.
\end{array}
\right.
\]
\end{example}

By the above example we get, that the lower bound for the complex Vasquez invariant is sharp. The following one shows that one of the upper bound also has this property.

\begin{proposition}
Let $G$ be an elementary abelian $2$-group of rank $k \geq 2$. Then
\[
n_\CC(G) = \frac{1}{2}\left(n(G)+\sum_{\chi \in \Irr_1(G)}m_\Q(\chi)\chi(1)\right) = 2^{k-1}+2^{k-2}(2^k-1).
\]
\end{proposition}

\begin{proof}
Let $\X$ denote the set of non-trivial elements of $G$. 
Let 
\[
S=\bigoplus_{a \in \X}\ind_{\langle a \rangle}^G\Z.
\]
By \cite[Theorem 2]{CW89} there exists special element $\alpha \in H^2(G,S)$ with the property that 
\begin{equation}
\label{equation:not-special}
\nu_*(\alpha) \text{ is not special}
\end{equation}
for any non-trivial submodule $M$ of $S$, where $\nu \colon S \to S/M$ is the natural mapping. In particular
\[
n(G) = \rk_\Z(S) = 2^{k-1}(2^k-1).
\]

For every $a \in \X$ let $\chi_a$ denote the trivial character of the group $\langle a \rangle$. The character $\chi_S$ of the $G$-module $S$ is given by the formula
\[
\chi_S = \sum_{a \in \X} \ind_{\langle a \rangle}^G \chi_a.
\]

Let $\chi_G$ be the trivial character of $G$. If by $\cal K$ we denote the set of subgroups of $G$ of index $2$, then
\[
\Irr(G) = \Irr_1(G) = \{\chi_G\} \cup \{ \chi_K : K \in \cal K \},
\]
where $\chi_K$ is the irreducible character of $G$ with kernel $K \in \cal K$.

By the Frobenius reciprocity, we get
\[
(\chi_G, \chi_S) = \sum_{a \in \X} (\chi_G,\ind_{\langle a \rangle}^G \chi_a) = \sum_{a \in \X} (\res_{\langle a \rangle}\chi_G,\chi_a) = |\X| = 2^k-1
\]
and for $K \in \cal K$
\[
(\chi_K, \chi_S) = \sum_{a \in \X} (\chi_K,\ind_{\langle a \rangle}^G \chi_a) = \sum_{a \in \X} (\res_{\langle a \rangle}\chi_K,\chi_a) = |K|-1=2^{k-1}-1.
\]
The above calculations show that
\[
\chi_S = (2^k-1)\chi_G + \sum_{K \in \cal K} (2^{k-1}-1) \chi_K,
\]
hence $S$ does not admit any complex structure, but if $R$ is the regular $G$-module, then
\[
L := S \oplus R,
\]
with character
\[
\chi_L = 2^k\chi_G + \sum_{K \in \cal K} 2^{k-1} \chi_K
\]
is an essentially complex $G$-lattice. The cohomology class
\[
\alpha = \alpha + 0 \in H^2(G,S) \oplus H^2(G,R)
\]
is of course special.

Assume that $M \neq 0$ is such a $G$-submodule of $L$ that $L/M$ is essentially complex and the image of $\alpha$ in $H^2(G,L/M)$ is special. Obviously, $M$ is essentially complex itself. Arguing as in the proof of Lemma \ref{lemma:lower_bound}, we get that
\[
M \cap S \neq 0
\]
is a $\Z$-pure submodule of $S$ and the projection of $\alpha$ to $H^2(G,S/M \cap S)$ gives a special element, which contradicts property \eqref{equation:not-special}.

To sum up, $L$ and $\alpha \in H^2(G,L)$ define an aspherical K\"ahler group $\Gamma'$, of dimension $2^{k-1}+2^{k-2}(2^k-1)$, such that for every non-trivial and normal subgroup $M$ of $\Gamma'$ with the property that $M \subset L$, $\Gamma'/M$ is not an aspherical K\"ahler group. This shows that
\[
n_\CC(G) \geq 2^{k-1}+2^{k-2}(2^k-1) = \frac{1}{2}\left(n(G)+\sum_{\chi \in \Irr_1(G)}m_\Q(\chi)\chi(1)\right).
\]
Applying Proposition \ref{proposition:bounds}.\ref{proposition:bounds:schur} finishes the proof.
\end{proof}

\section{Holomorphic maps}

From the algebraic point of view what we have got so far is an epimorphism of one fundamental group to another. This gives us a \emph{continuous} map between complex manifolds. In this section, we will investigate the situation in which the induced map is in fact holomorphic.

Let us start, as usual, with $G$-lattice $L$ and its $\Z$-pure sublattice $M$. The natural homomorphism induces an $\R G$-epimorphism $f \colon M^\R \to (M/L)^\R$, given by
\[
(1\otimes l) \mapsto 1\otimes (l+M)
\]
for $l \in L$. Note that we skip the subscript $\Z$ in the notation of the tensor product.

\begin{lemma}
$\ker f = M^\R$.
\end{lemma}
\begin{proof}
Take $m \in M$. Since
\[
f(1 \otimes m) = 1 \otimes (m+M) = 0,
\]
hence $M^\R \subset \ker f$. We go to the conclusion, noting that
\[
\dim_\R M^\R +\dim_\R (L/M)^\R = \dim_\R L^\R.
\]
\end{proof}

Assume that $L$ and $L/M$ are essentially complex with complex structures $J$ and $J'$ respectively. 
\begin{definition}
The map $f$ is \emph{holomorphic} if 
\[
fJ=J'f.
\]
\end{definition}

\begin{lemma} Assume that $f$ is holomorphic. Then:
\label{lemma:properties-of-holomorphic-map}
\begin{enumerate}
	\item Kernel of $f$ is $J$-invariant.
	\item $J'$ is uniquely determined by $J$.
\end{enumerate}
\end{lemma}

\begin{proof}
Let $v \in \ker f$. Then
\[
fJ(v) = J'f(v) = J' 0 = 0
\]
and $J(v) \in \ker f$. Now take $w \in (L/M)^\R$ and $v \in L^\R$ such that $f(v) = w$. We have
\[
J'(w) = J'f(v) = fJ(v).
\]
\end{proof}

\begin{corollary}
\label{corollary:holomorphic-iff-condition}
The map $f$ is holomorphic if and only if $\ker f$ is $J$-invariant.
\end{corollary}

\begin{proof}
Existence of $J'$ is given by the isomorphism theorem. Moreover, for $w = f(v) \in (L/M)^\R$ we have
\[
J'^2(w)= J'^2f(v)= J'fJ(v) fJ^2(v) = f(-v) = -f(v) = -w.
\]
\end{proof}

The following example shows that the algebraic construction given by Theorem \ref{theorem:complex-vasquez} does not give us holomorphic maps in general. We deal with this problem in the next section.

\begin{example}
Let the Bieberbach group $\Gamma\subset\operatorname{Iso}(\R^6)$ be generated by
\[
(I,e_1),\ldots,(I,e_6),\left(-I_2\oplus I_4,\frac12e_6\right).
\]
$\Gamma$ fits into the short exact sequence
\[
0 \to \Z^6 \to \Gamma \to C_2 \to 1.
\]
Since the holonomy representation of $\Gamma$ splits into two homogeneous components, every complex structure on $\Gamma$ will be a direct sum of the two ones, corresponding to the splitting. Let us focus on the four-dimensional part of the representation, since the other one is quotient out by the construction.
Using the notation of the proof of Theorem \ref{thm:vasquez}, we have $L = \Z^4$ with the trivial $G$-action, $L_0 = \{ (0,0,0,d)^T : d \in \Z \}$ and $L_0' = \{ (a,b,c,0)^T : a,b,c \in \Z \}$. Assume that $L$ has the following complex structure:
\[
J = \begin{bmatrix}
-1&-\frac{\sqrt3}{3}&-1&0\\
3+\sqrt3&0&0&\sqrt3\\
1-\sqrt3&\frac{\sqrt3}{3}&1&-1\\
1+\sqrt3&1&1+\sqrt3&0
\end{bmatrix}.
\]
Our goal is to find a rank $2$ submodule $M$ of $L_0'$ such that $M^\R$ is $J$-invariant. Assume that such an $M$ exists and $v = (a,b,c,0)^T$ is a non-zero element of $M$. We get
\[
\def\arraystretch{1.5}
Jv=\begin{bmatrix}
-\frac{1}{3}(3a+\sqrt3b+3c)\\
(3+\sqrt3)a\\ 
\frac{1}{3}((3-3\sqrt3)a+\sqrt3+3c)\\
(1+\sqrt3)a+b+(1+\sqrt3)c
\end{bmatrix}.
\]
Since the last coordinate of $Jv$ is zero it follows that $b = a+c = 0$, $v = (a,0,-a,0)$ and $M$ is of rank at most $1$, a contradiction.
\end{example}

\section{Changing complex structures}

As in the previous sections, denote by $K$ the ring of integers or the field of real or complex numbers. Let $V$ be a $K G$-module with a complex structure $J$. We have obvious decomposition of the $\CC G$-module $V^{\CC}$
\[
V^{\CC} = V_J^{1,0} \oplus V_J^{0,1},
\]
where $V_J^{1,0}$ and $V_J^{0,1}$ are the eigenspaces of the action of $J$ with the eigenvalues $i$ and $-i$ respectively. Motivating by geometry, we will call them \emph{holomorphic} and \emph{anti-holomorphic} parts of $V$. In the case when $J$ is fixed we will drop the subscript in the notation.

Note that if we have two complex structures on $V$, say $J$ and $J'$, then $V_J^{1,0}$ and $V_{J'}^{1,0}$ are isomorphic as $\R G$-modules. However, they may be non-isomorphic as $\CC G$-modules. We will deal with this problem now in the case $K=\R$. Note that it is enough to consider homogeneous modules, since they are preserved by all complex structures.

\begin{lemma}
\label{lemma:isomorphism-of-holomorphic-parts}
Let $V$ be a homogeneous $\R G$-module with complex structures $J$ and $J'$. Then $V_J^{1,0}$ and $V_{J'}^{1,0}$ are isomorphic in one of the following cases:
\begin{enumerate}[label=(\alph*)]
    \item the simple component of $V$ is of real or quaternionic type;
    \item the complex structures are conjugated in $\Aut_{\R G}(V)$.
\end{enumerate}
\end{lemma}

\begin{proof}
The case (a) is obvious, since then $V^\CC$ is homogeneous. Now assume that $J' = AJA^{-1}$ for some $A \in \Aut_{\R G}(V)$. It is easy to check that $V_{J'}^{1,0} = A V_J^{1,0}$. 
\end{proof}

\begin{remark}
\label{remark:conjugacy_of_complex_structures}
Note that in the case (b) of the above lemma it is somehow easy to determine the conjugacy class of a complex structure $J$ in $\Aut_{\R G}(V)$, since if $S$ is a simple component of $V$ and it is of multiplicity $n$, then
\[
\Aut_{\R G}(V) \cong \GL_n(\End_{\R G}(S)) \cong \GL_n(\CC).
\]
Identifying $J$ as an element of $\GL_n(\CC)$ it is enough to count its eigenvalues, which are of course $\pm i$.
\end{remark}

\begin{proposition}
\label{proposition:invariance-under-complex-struture}
Let $V$ be a homogeneous $\R G$-module with a complex structure $J$. Let $W$ be an essentially complex submodule of $V$. There exists a complex structure $J'$, such that $J'W = W$ and $V_J^{1,0} \cong V_{J'}^{1,0}$.
\end{proposition}
\begin{proof}
Let $V = W\oplus W'$ and let $S$ be a simple component of $V$. If $S$ is of real or quaternionic type, by Lemma \ref{lemma:isomorphism-of-holomorphic-parts} it is enough to take $J' = J_W \oplus J_{W'}$ where $J_W$ and $J_{W'}$ are any complex structures on $W$ and $W'$ respectively.

Assume that $S$ is of complex type. Denote by $n$ and $k$ the multiplicity of $S$ in $V$ and $W$ respectively. We get that
\begin{equation}
\label{eq:invariant-under-j}
W = \bigoplus_{i=1}^k S \text{ and } W' = \bigoplus_{i=k+1}^n S.
\end{equation}
Identifying $\Aut_{\R G}(V)$ with $\GL_n(\CC)$ as in Remark \ref{remark:conjugacy_of_complex_structures}, assume that a Jordan form of $J$ is $\diag(a_1,\ldots,a_n)$, where $a_i = \pm i$ for $i=1,\ldots,n$. By the same remark and the form \eqref{eq:invariant-under-j} of $W$ and $W'$ it is enough to take $J'=\diag(a_1,\ldots,a_n)$.
\end{proof}

\begin{corollary}
\label{corollary:good-complex-structure}
Let $L$ be a $G$-module with a complex structure $J$. Assume that $M$ is an essentially complex $\Z$-pure sublattice of $L$. There exists a complex structure $J'$ on $L$ such that $M^\R$ is $J'$ invariant and $L_J^{1,0} \cong L_{J'}^{1,0}$. In particular we have $M_{J'}^{1,0} \subset L_{J'}^{1,0}$.
\end{corollary}

\begin{proof}
Let 
\[
L^\R = L_1 \oplus \ldots \oplus L_k
\]
be the decomposition into homogeneous components. We get that
\[
M^\R = M_1 \oplus \ldots \oplus M_k
\]
is also a decomposition into homogeneous components and
\[
J = J_1 \oplus \ldots \oplus J_k,
\]
where for every $1 \leq i \leq k$, $J_i$ is a complex structure of $L_i$ and $M_i  = M^\R \cap L_i$. By Proposition \ref{proposition:invariance-under-complex-struture} for every $1 \leq i \leq k$ there exists a complex structure $J_i'$ of $L_i$, giving isomorphic holomorphic part and for which $J_i' M_i = M_i$. Taking $J' = J_1' \oplus \ldots \oplus J_k'$ and observing that 
\[
L_J^{1,0} = (L_1)_{J_1}^{1,0} \oplus \ldots \oplus (L_k)_{J_k}^{1,0}
\]
(similar for $J'$) we get desired result.
\end{proof}

\section{Holomorphic tangent bundles}

Let an aspherical K\"ahler group $\Gamma$ of dimension $n$ be given by the short exact sequence \eqref{eq:exact} and $X = \R^{2n}/\Gamma$. By proof of \cite[Proposition 1.1]{V70} the tangent bundle of $X$ is given by
\[
TX = (\tilde X \times L^\R)/\Gamma,
\]
where the action of $\Gamma$ on $\tilde X \times L^\R$ is given by
\[
\gamma (x,v) = (\gamma \cdot x, d\gamma \cdot v),
\]
for $\gamma \in \Gamma, x \in \tilde X, v \in L^\R$.
Note that the universal cover $\tilde X$ equals $\R^{2n}$ and for $\gamma = (A,a) \subset \GL(2n,\R) \ltimes \R^{2n}, d\gamma = A$, so the action of $\Gamma$ on $L^\R$ comes exactly from the $G$-module $L$. Let $J$ be a complex structure on $X$. Denote by $X_J$ the corresponding generalized hyperelliptic manifold. By \cite[Proposition 2.6.4]{Hu05}, up to isomorphism of complex vector bundles, the holomorphic tangent bundle of $X_J$ is given by
\[
TX_J^{1,0} = (\tilde X \times L_J^{1,0})/\Gamma.
\]

Let $M \subset L$ be an essentially complex submodule such that $\Delta = \Gamma/M$ is torsion-free. By \cite[Main Theorem 2.3]{V70} we get a submersion $f \colon X \to Y$, given by the natural homomorphism $\Gamma \to \Delta$, where $\pi_1(Y) = \Delta$. Moreover, by \cite[Lemmas 2.6 and 2.7]{V70} we have the short exact sequence of real vector bundles
\[
0 \longrightarrow \ker \rho \longrightarrow TX \stackrel{\rho}{\longrightarrow} f^*(TY) \longrightarrow 0,
\]
where $\ker \rho = (\tilde X \times M^\R)/\Gamma$ is a pullback of a vector bundle over $Y$.

\begin{remark}
Note that if $f$ is holomorphic, then by Lemma \ref{lemma:properties-of-holomorphic-map} the complex structure on $Y$ is fixed. For the sake of making notation as clear as possible we will not give any new symbol to it.
\end{remark}

\begin{theorem}
Let $X$ be a flat manifold with a complex structure $J$. There exists a complex structure $J'$ on X and a compact flat K\"ahler manifold $Y$ such that the following hold:
\begin{enumerate}
    \item There exist a holomorphic submersion $f \colon X_{J'} \to Y$.
    \item $TX_J^{1,0}$ and $TX_{J'}^{1,0}$ are isomorphic complex vector bundles.
    \item $TX_{J'}^{1,0}$ is isomorphic to a pullback of a complex vector bundle over $Y$. 
\end{enumerate}
\end{theorem}

\begin{proof}
We will keep the notation of the discussion preceding the statement of the theorem. By Corollary \ref{corollary:holomorphic-iff-condition}, the map $f$ will be holomorphic for such a complex structure $J'$ that $M^\R$ is $J'$-invariant. Using Corollary \ref{corollary:good-complex-structure} we get that not only such $J'$ exists, but it gives us isomorphism of $TX_J^{1,0}$ and $TX_{J'}^{1,0}$. Moreover, since $M_{J'}^{1,0} \subset L_{J'}^{1,0}$, we have a short exact sequence of complex vector bundles
\[
0 \longrightarrow \ker \pi \longrightarrow TX_{J'}^{1,0} \stackrel{\pi}{\longrightarrow} f^*(TY^{1,0}) \longrightarrow 0,
\]
where $\ker\pi = (\tilde X \times M_{J'}^{1,0})/\Gamma = \ker \rho^{1,0}$, hence it is a pullback of some complex vector bundle over $Y$. By the construction, $TX_{J'}^{1,0}$ is isomorphic to the vector bundle $(\tilde X \times M^{1,0} \oplus L^{1,0}/M^{1,0})/\Gamma$ which is exactly the Whitney sum of $\ker \pi$ and $f^*(TY^{1,0})$. This finishes the proof.
\end{proof}

The proof of the second main theorem of the paper is now formality.

\begin{proof}[Proof of Theorem \ref{theorem:chern-classes}]
Keeping notation of this section, for $i \in \N$ we have
\[
c_i(X) = f^*(c_i(E)),
\]
where $E$ is a complex vector bundle over $Y$. For $i > n_{\CC}(G) \geq \dim_{\CC}(Y)$ we have
\[
c_i(E) \in H^{2i}(Y,\Z) = 0
\]
and the result follows.
\end{proof}

\begin{remark}
In the language of \cite{CC17} we say that $L$ with complex structures $J$ and $J'$ have the same Hodge type. Using description of the space of complex structures on $X$, given for example in \cite{BR11}, we can show more: $J'$ may be constructed in a way that there is a continuous path of complex structures on $X$ between $J$ and $J'$. Using deformation theory gives another way of showing that $TX_J^{1,0}$ and $TX_{J'}^{1,0}$ are isomorphic.
\end{remark}

\section*{Acknowledgments}

The authors would like to thank Andrzej Szczepański for helpful discussions.

\printbibliography

\noindent
Anna Gąsior\\
Institute of Mathematics, Maria Curie-Sk{\l}odowska University\\
Pl. Marii Curie-Sk{\l}odowskiej 1\\
20-031 Lublin\\
Poland\\
E-mail: \texttt{anna.gasior@mail.umcs.pl}
\vskip 5mm
\noindent
Rafał Lutowski\\
Institute of Mathematics, Faculty of Mathematics, Physics and Informatics, University of Gda\'nsk\\
ul. Wita Stwosza 57\\
80-952 Gda\'nsk\\
Poland\\
E-mail: \texttt{rafal.lutowski@ug.edu.pl}

\end{document}